\documentclass[12pt,twoside]{amsart}
\usepackage{amsmath}
\usepackage{amssymb}
\usepackage{graphicx}

\def\sideremark#1{\ifvmode\leavevmode\fi\vadjust{\vbox to0pt{\vss 
      \hbox to 0pt{\hskip\hsize\hskip1em           
 \vbox{\hsize3cm\tiny\raggedright\pretolerance10000
 \noindent #1\hfill}\hss}\vbox to8pt{\vfil}\vss}}}%

\newtheorem{theorem}{Theorem}
\newtheorem{corollary}{Corollary}
\newtheorem{proposition}{Proposition}
\newtheorem{lemma}{Lemma}

\theoremstyle{definition}
\newtheorem{definition}{Definition}
\newtheorem{example}{Example}
\theoremstyle{remark}
\newtheorem{remark}{Remark}

\renewcommand{\epsilon}{\varepsilon}

\newcommand{\var}{\mathcal{V}ar}

\newcommand{\R}{\mathbb{R}}
\newcommand{\C}{\mathbb{C}}

\def\const{\operatorname{const}}
\def\dist{\operatorname{dist}}
\def\Arg{\operatorname{Arg}}

\newcommand{\dr}{\mathrm{d}}
\newcommand{\bbC}{\mathbb{C}}

\newcommand{\diff}{\Delta}
\newcommand{\hp}{\mathcal{H}}
\DeclareMathOperator{\re}{\mathrm{Re}}
\DeclareMathOperator{\im}{\mathrm{Im}}

\newcommand{\bbr}{\mathbb{R}}
\newcommand{\bbc}{\mathbb{C}}
\newcommand{\bbz}{\mathbb{Z}}
\newcommand{\ep}{\epsilon}

\title[Codimension one exponential case]{Pseudo-Abelian integrals: unfolding generic exponential case}
\author{Marcin Bobie\'nski}
\address{Institute of Mathematics, Warsaw University,
ul. Banacha 2, 02-097 Warsaw, Poland} \email{mbobi@mimuw.edu.pl}
\thanks{This research was supported by
KBN Grant No 2 P03A 015 29, Conseil Regional de Bourgogne 2006, (No
05514AA010S4115) and Soref New Scientists Start up Fund Fusfeld
Research Fund}

\author{Pavao Marde\v{s}i\'c}
\address{Universit\'e de Bourgogne, Institut de
Math\'ematiques de Bourgogne, U.M.R. 5584 du C.N.R.S., B.P. 47870,
21078 DIJON CEDEX - FRANCE} \email{mardesic@u-bourgogne.fr}

\author{Dmitry Novikov}
\address{Department of Mathematics, Weizmann Institute of Science,
Rehovot, Israel}\email{dmitry.novikov@weizmann.ac.il}
\begin{document}
\today \maketitle

\begin{abstract}
We consider functions of the form $H_0=P^{a_1}\cdots P_k^{a_k}e^{R/Q}$, with $P_i,$ $R$, and $Q\in\R[x,y]$, which are (generalized Darboux) first integrals of the polynomial system $Md\log H_0=0$.  We assume that $H_0$ defines a family $\gamma(h)\subset H_0^{-1}(h)$ of real cycles in a region bounded by a polycycle.

To each polynomial form $\eta$ one can associate the pseudo-abelian integrals $I(h)$ of $M^{-1}\eta$ along $\gamma(h)$, which is the first order term of the displacement function of the orbits of $MdH_0+\delta\eta=0$.

We consider Darboux first integrals unfolding $H_0$ (and its saddle-nodes) and pseudo-abelian integrals associated to these unfoldings. Under genericity assumptions we show the existence of a uniform local bound for the number of zeros of these pseudo-abelian integrals.

The result is part of a program
to extend Varchenko-Khovanskii's theorem from abelian integrals to pseudo-abelian integrals and prove the existence of a bound for the number of their zeros in function of the degree of the polynomial system only.
\end{abstract}

\section{Introduction and main Results}
This paper is a
part of a program for generalizing the results of Varchenko and
Khovanskii \cite{v,kh} giving the boundednes of the number of zeros
$A(n)$ of Abelian integrals corresponding to polynomial deformations
of degree $n$ of Hamiltonian vector fields. We want to generalize
this result to deformations of polynomial Darboux integrable
systems. The general strategy as in \cite{v,kh} is to prove local
boundednes and use compactness of the product of the parameter space
by the limit periodic sets (see also Roussarie \cite{r2}). In
previous papers \cite{N}, \cite{BM} we proved local boundednes of
the number of zeros of pseudo-abelian integrals under generic
hypothesis. We prove here an analogous result in one of the first
non-generic cases where an exponential factor appears in the
first integral. Generically, in the unfolding two
invariant algebraic curves bifurcate from the exponential factor in
saddle-node bifurcations. Other nongeneric cases have been studied in \cite{B} and \cite{NN}

Consider a real rational closed meromorphic one-form $\theta_0$
having a generalized Darboux first integral of the form
\begin{equation}\label{eq:theta0} H_0=P_1^{a_1}\cdots
P_k^{a_k}e^{R/Q},\quad \theta_0=d(\log H_0).
\end{equation}

Choose a limit periodic set i.e. bounded component of $\R^2\setminus \{Q\prod P_i=0\}$
filled by cycles $\gamma(h)\subset\{H_0=h\}$, $h\in(0,b)$. Denote by $D\subset H^{-1}(0)$ the polycycle which is in
the boundary of this limit periodic set. The other component of the boundary of the limit periodic set belongs to $H^{-1}(b)$.

Let $U^\R$ be a neighborhood of
$D$ in $\R^2$, and let $U$ be a neighborhood of $D$ in $\C^2$.

We assume that $Q^{-1}(0)$ contains one or more edges of $D$. If the curve
$Q^{-1}(0)$ does not cut the polycycle $D$, then the first integral has a
form $H=f^*\, \prod P_i^{a_i}$, where $f^*$ is a non-vanishing
holomorphic function in a neighborhood of the polycycle and the
proof in \cite{N} or \cite{BM} goes through without any
modification. Note that the assumption that the curve $Q^{-1}(0)$ cuts the
polycycle $D$ implies that $R^{-1}(0)\cap Q^{-1}(0)=\emptyset$.
Indeed, in a neighborhood of any (transversal) intersection point
$p\in R^{-1}(0)\cap Q^{-1}(0)$ the first integral function reads
$H=e^{x/y}$ and so the point $(0,0)$ does not belong to the closure
of a bounded region filled with closed orbits $\gamma(h)$.

Denote the union of the edges of $D$ lying in $Q^{-1}(0)$ by $L_E$. Each of the
vertices of $D$ lying on $L_E$ is a saddle-node and $L_E$ lies in the strong variety of these saddle-nodes. [[see picture 1a]]

We assume that the form $\theta_0$ is generic:
\begin{definition}
\label{genericity} Denote $L_i^\R=P_i^{-1}(0),$ $L_E^\R=Q^{-1}(0)$
and $L_i^\C$ and $L_E^\C$ their complexification.
We assume that the following properties are satisfied by $\theta_0$
in the neighborhood $U$ of the polycycle $D$:
\begin{enumerate}
\item{} the curves $P_j^{-1}(0)$,   $Q^{-1}(0)$ are smooth and reduced,
\item{} $P_i^{-1}(0)$ and $P_j^{-1}(0)$, as well as
    $Q^{-1}(0)$ and $P_j^{-1}(0)$ intersect transversally.
\end{enumerate}
\end{definition}

Consider an unfolding $\theta_{\epsilon,\alpha}$ of the form
$\theta_0$, where $\theta_{\epsilon,\alpha}$ are real rational
closed one-forms with the Darboux first integral
\begin{equation}\label{eq:ea family}
H_{\epsilon,\alpha}=P_1^{a_1}\cdots
P_k^{a_k}Q^{\alpha-1/\epsilon}(Q+\epsilon R)^{1/\epsilon},\quad
\theta_{\epsilon,\alpha}=d(\log H_{\epsilon,\alpha}).\end{equation}

\begin{figure}[htp]
\begin{center}
\includegraphics[width=0.9\hsize]{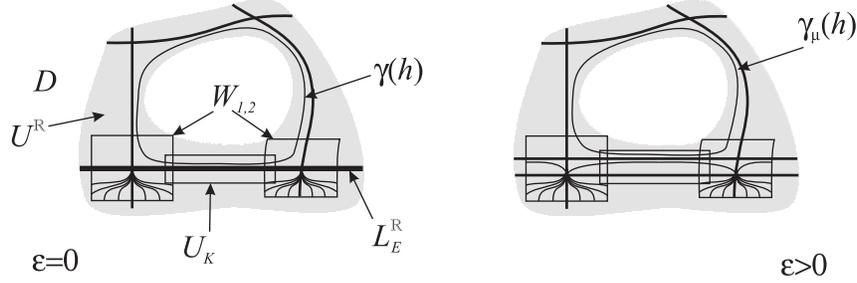}
\end{center}
\caption{Unfolding of the polycycle $D$}\label{fig:Dunfolding}
\end{figure}

The foliation defined by $\theta_{\epsilon, \alpha}$ has a maximal
nest of cycles
$\gamma_{\epsilon,\alpha}(h)\subset\{H_{\epsilon,\alpha}=h\}$,
$h\in(0,b(\epsilon,\alpha))$ filling a connected component of
$\R^2\setminus\{Q(Q+\epsilon R)\prod P_i=0\}$ whose boundary is a
polycycle  $D_{\epsilon,\alpha}$ close to $D$.
Consider pseudo-abelian integrals of the form
\begin{equation}
\label{eq:integral}
I_{\epsilon,\alpha}(h)=\int_{\gamma_{\epsilon,\alpha}(h)}M^{-1}\eta,
\quad\text{where} \quad M=Q (Q+\epsilon R)\prod_{i=1}^k P_i
\end{equation}
and $\eta$ is a polynomial one-form of degree at most $n$.

This integral appears as the linear term with respect to $\delta$ of
the displacement function of a polynomial deformation
\begin{equation}
\label{polfam} M\theta_{\epsilon,\alpha}+\delta\eta=0
\end{equation}
of the Darboux integrable polynomial vector field with the first
integral $H_{\epsilon,\alpha}$, see (\cite{BM} and \cite{N}).

\begin{theorem}\label{main}
Under the genericity assumptions of Definition \ref{genericity} we
have that
the number of isolated zeros of pseudo-abelian integrals
$I_{\epsilon,\alpha}$ in their maximal interval of definition
$(0,b(\epsilon,\alpha))$ is locally uniformly bounded.

More precisely, for any $n$ there exist an $\epsilon_0>0$ and an
upper bound $N$, depending on $\theta_0$ and $n$ only, such that for
any $|\epsilon|,|\alpha|<\epsilon_0$ and any $\eta$, $\deg\eta\le
n$, the number of isolated zeros of pseudo-abelian integral
(\ref{eq:integral}) in $(0,b(\epsilon,\alpha))$ is at most $N$.
\end{theorem}

In fact, by  Varchenko-Khovanskii's theorem \cite{v,kh} the number of zeros of
$I(h)$ in any interval $[r, b(\epsilon,\alpha))$ is locally
uniformly bounded for any $r>0$. That is the only point that has to be proved is the local boundedness of the number of zeros of pseudo-abelian integrals
in some interval $(0,r)$, for $r>0$ sufficiently small, i.e. for values corresponding to a neighborhood of the polycycle $D$.

Following long tradition of \cite{ilyashenko, ecalle}, we completely abandon polynomial settings for analytic ones, and prove more general Theorem~\ref{main:localan} below. Theorem~\ref{main:localan} deals with unfoldings of a real analytic integrable foliation defined in a neighborhood of the polycycle $D$ and claims that, assuming local analytic analogues of conditions of Theorem~\ref{main},  the number of zeros of corresponding pseudoabelian integrals is locally uniformly bounded. Theorem~\ref{main} follows from this as indicated above.

Let $\theta_0$ be  a closed meromorphic one-form defined in a topological annulus $U^\R\subset\R^2$ and satisfying the following conditions:
\begin{itemize}
\item $\theta_0=d(\frac R S)+\sum a_i\frac{dP_i}{P_i}+\theta',$ where $R,S$ and $P_i$ are analytic in $U^\R$, and $\theta'$ is a closed one-form analytic in $U^\R$;
\item $P^{-1}(0), S^{-1}(0)$ are smooth, reduced and intersect transversally.
\end{itemize}
We assume that the foliation defined by $\theta_0$ in $U^\R$ has a nest of cycles accumulating to a polycycle $D\subset U^\R$ lying in a polar locus of $\theta_0$, and let $U^\R$ be a sufficiently small neighborhood of $D$. This in particular implies that $\theta'=df^*$ for some analytic in $U^\R$ function $f^*$, which can be further assumed to be equal to zero (by changing $P_1$ to $P_1\exp(f^*/a_1)$). We assume that some edges of $D$ lie on $\{Q=0\}$, as the other case was considered before \cite{BM,N}.

Consider a finite-dimensional analytic  (with topology of uniform convergence on compact sets) family $\Theta$ of pairs $\left(\theta_{\mu},\eta_{\mu}\right)$ of one-forms defined in a complex neighborhood $U$ of the polycycle $D$, $\mu\in \R^m$. We assume that $\theta_{\mu}$ is a real meromorphic closed one-form  and $\eta_\mu$ is real holomorphic one-form in $U$.

Assume that  the polar locus $D_{\mu}$ of $\theta_{\mu}$ is a union of deformations of components of $D$: this means that the forms $Q_{1,\mu}Q_{2,\mu}P_{1,\mu}...P_{k,\mu}\theta_\mu$ are holomorphic one-forms on $U$, where $Q_{1,\mu},Q_{2,\mu}$ and $P_{i,\mu}$ are analytic in $\mu$ families of real holomorphic functions defined in $U$, with $Q_{1,0}=Q_{2,0}=Q$ and $P_{i,0}=P_i$. The function $M_\mu=Q_{1,\mu}Q_{2,\mu}P_{1,\mu}...P_{k,\mu}$ will be called the \emph{integrating factor} of $\theta_{\mu}$.

Assume moreover that the real foliations defined by $\theta_{\mu}$ have nests of cycles $\gamma_{\mu}(h)\subset\{H_\mu=h\}$ accumulating to $D_{\mu}$, where $H_\mu$ is the first integral of the foliation defined by $\theta_\mu$, namely $H_{\mu}=\exp\left(\int\theta_\mu\right)$.

\begin{theorem}\label{main:localan}
There exists $r>0$ such that the number of zeros of the pseudo-abelian integral
$$I_{\mu}(h)=\int_{\gamma_{\mu}(h)}M^{-1}\eta_\mu$$
in $(0,r)$ is uniformly bounded over all $\mu$ in a sufficiently small neighborhood of $0$ in $\R^m$.
\end{theorem}

\begin{example} The family (\ref{eq:ea family}) satisfies conditions of Theorem~\ref{main:localan}: in this case $\mu=(\epsilon,\alpha)$, $Q_{1,\mu}$ and $P_{i,\mu}$ do not depend on $\mu$ and $Q_{2,\mu}=Q+\epsilon R$.
\end{example}

\section{Plan of the proof.}

\subsection{Analytic continuation of pseudo-Abelian integral.}
The first step is to show that the integral $I_{\mu}(h)$ can be
analytically continued to the universal cover of the punctured
disc $\{0<|h|<r\}$ for some sufficiently small $r$. As in \cite{BM}, this is obtained by
transporting the cycle of integration to nearby leaves. More precisely,
in a complex neighborhood of the polycycle $D$ we construct two linearly independent real
vector fields preserving the foliation and transversal to it. This allows to define lifting of
vector fields from a punctured neighborhood of zero in $\C_h$  to the neighborhood $U$ of $D$ as linear combinations of these vector fields, see
Section~\ref{sec:transport}. We transport the real cycles $\gamma_\mu(h)$  using flows of these liftings.
\begin{remark}
Our  construction of local transport of cycles differs from the one
used in \cite{Paul}. Both constructions start from local vector
fields (so-called "Clemens symmetries"), and then use partition of
unity to get a transport defined in a neighborhood of $D$. However,
we glue together the vector fields themselves, and not their flows
as in \cite{Paul}.
\end{remark}

\subsection{Variation relation}
The form $\theta_\mu$ has a first order pole on $P_{j,\mu}^{-1}(0)$, so from closedness of $\theta_\mu$ it follows that the residue of $\theta_\mu$ on $P_{j,\mu}^{-1}(0)$ is well defined. We will denote it by $a_{j,\mu}$.

The main feature of the constructed transport is that the lifting
of $ih\partial_h$ is $2\pi a_{j,\mu}$-periodic in a neighborhood of
separatrics lying on $\{P_{j,\mu}=0\}$. This implies that the  cycle
$\gamma_{\mu}(h)\subset \{H_\mu=h\}$ and its transport to
$\gamma_{\mu}(he^{2\pi ia_{j,\mu}})\subset\{H_{\mu}=he^{2\pi ia_{j,\mu}}\}$ coincide in this
neighborhood, so the difference
$\gamma_{\mu}(he^{\pi ia_{j,\mu}})-\gamma_{\mu}(he^{-\pi ia_{j,\mu}})$ does not intersect a
neighborhood of $\{P_{j,\mu}=0\}$.

For pseudo-Abelian integrals this
geometric observation translates into the following construction. Define
the variation operator $Var_a$ as the difference between
counterclockwise and clockwise continuation of $I_{\mu}(h)$:
\begin{equation}
\label{vardef} Var_{a}(I_{\mu})(h)=I_{\mu}(he^{ia \pi})-I_{\mu}(he^{-i a \pi}),
\end{equation}
and denote by $Var_{a_1,...,a_k}$ the composition $Var_{a_{1,\mu}}\circ
\dots\circ Var_{a_{k,\mu}}$.

The key of the proof \cite{BM,N} of the local boundedness of the number of zeros of a generic Darboux
integrals on $H=P_1^{a_1}\cdot...\cdot P_k^{a_k}P_k^{a_{k+1}}$  was a lemma stating  that
$Var_{{a_1},\cdots,{a_k},{a_{k+1}}}I(h)\equiv 0$.   The main result was then
deduced from this by induction observing  (via a generalization of
Petrov's trick) that the operators $Var_a$  reduce the number of
isolated zeros of pseudo-abelian integrals by a
constant locally bounded for any  analytic  family $\Theta$
Here Proposition
\ref{mainlemma} provides a suitable form of Petrov's trick.
. The vanishing of the iterated variation permitted to start the induction using Gabrielov's theorem.

In our present situation we dont know how to associate a variation to the edge correspondig to the exponential factor in the first integral ($Q=0$ in Theorem \ref{main} or $S=0$ in Theorem \ref{main:localan}). We consider only iterated variation $Var_{a_1,...,a_k}$ asociated to all other edges. The operator
$Var_{a_1,...,a_k}$  does not annihilate
completely the pseudo-abelian integral, but produces a
univalued function in a transverse parameter see Theorem \ref{var}. This transverse parameter is shown to be $-1/\omega$, where $\omega$ is a Pfaffian function
generalizing the classical Ecalle-Roussarie compensator.

More precisely, we define a compensator
$\omega(h,\epsilon,\alpha)$ by the following relation
\begin{equation}\label{omega}
\widetilde{H}(-\frac 1 {\omega(h,\epsilon,\alpha)},
\epsilon,\alpha)=h
\end{equation}
where
$$
\widetilde{H}(x, \epsilon,\alpha)=\begin{cases}
x^{\alpha}\left(\frac{x-\epsilon}{x}\right)^{1/\epsilon}, \quad
\text{for}\quad
\epsilon\neq0\\
x^{\alpha}e^{-1/x},\quad \text{for}\quad \epsilon=0.\end{cases}$$
$\omega(h,\epsilon,\alpha)$ is a Pfaffian function of $h$:
\begin{equation}\label{eq:omega is Pfaffian}
\frac{\alpha(-1-\epsilon\omega)+\omega}{\omega(1+\epsilon\omega)}d\omega=\frac{dh}{h}
\end{equation}

In
section \ref{sec:apn1} we prove existence of this function and
investigate its analytic properties. Note that
$\omega(h,\epsilon,0)$ is the usual Roussarie-Ecalle compensator,
i.e. $\omega(h,\epsilon,0)=\frac{h^{\epsilon}-1}{\epsilon}$, for
$\epsilon\neq 0$.

\begin{theorem}
\label{var}
For a pseudo-abelian integral $I_{\mu}(h)$ corresponding to the family
$\Theta$ there exist several pairs of real analytic functions $(\epsilon_i(\mu),
\alpha_i(\mu))$,  $\epsilon_i(0)=\alpha_i(0)=0$, such that
\begin{equation}\label{eq:thmvar}
Var_{{a_1},\cdots,{a_k}}(I_{\mu})(h)=\sum_{i=1}^N f_i(-\frac{1}{\omega(h,\epsilon_i(\mu),\alpha_i(\mu))},\epsilon_i(\mu),\alpha_i(\mu),\mu),
\end{equation}
where $f_i(u,\epsilon,\alpha,\mu)$ are meromorphic in $u$ in some small
disc and depends analytically on $\epsilon,\alpha\mu$ varying in some
small bidisc near the origin in $\R_{(\epsilon,\alpha)}\times\R_{\mu}$.
\end{theorem}

\begin{example} It will follow from the proofs that the number $N$ of such pairs is at most the number of arcs of $D$  lying on $\{Q=0\}$. However, for the family (\ref{eq:ea family})
there is only one pair of  parameters $\epsilon_i,\alpha_i$ in (\ref{eq:thmvar}) coinciding with the parameters $\epsilon,\alpha$ of the family.\end{example}

\subsection{End of the proof: application of Petrov trick}
Fewnomials theory of Khovanskii enables us to start the proof by induction.  It gives that the number of zeros of
the right-hand  side of (\ref{eq:thmvar})
on any interval $0\le u \le r$ is uniformly
bounded for all $\mu$ sufficiently small. Theorem \ref{main:localan} (and therefore Theorem \ref{main}) follow next by Petrov's argument, which allows to estimate the number of real
zeros of $J$ in terms of the number of zeros of $Var_aJ$, see
Lemma \ref{mainlemma}. The key technical difficulty is to prove existence
of a suitable asymptotic series for
$Var_{{a_1},\cdots,{a_k}}(I_{\mu})(h)$, see Proposition
\ref{pr:zersing}, which allows to translate a priori estimates on
the growth of the pseudo-abelian integral $I_{\mu}(h)$
 to estimates on variation of  its argument along small
arcs.

\section{Transport of cycles near the polycycle}\label{sec:transport}

In this section we construct a pair $v^{\mu}=(v^{\mu}_1,v^{\mu}_I)$ of two smooth real
vector fields defined in some complex neighborhood $U$ of the
polycycle $D$, analytically depending on $\mu$ and satisfying
\begin{equation}\label{eq:symms}
d(\log H_\mu)(v^{\mu}_1)=1, \qquad d(\log H_\mu)(v^{\mu}_I)=I,
\end{equation}
where, as before, $H_\mu=\exp(\int\theta_\mu)$.
Using these vector fields we can lift smooth curves $\varrho(t):[0,1]\to
\{0<|h|<h_0\}$ from  a small punctured disc $\{0<|h|<h_0\}$ to $U$, starting from any point of $H^{-1}(\varrho(0))\cap U$, provided that the lifted curve does not leave $U$.
We show that for $h_0$ small enough the lifting does not leave $U$ if the starting point of the lifting lies on the real cycle of integration $\gamma_{\mu}(h)$,
$h=\varrho(0)\in \R_+$. This allows to construct point-wise transport of
$\gamma_{\mu}(h)$ along any such curve $\varrho(t)$ by
transporting each point along its own lifting of the curve, and
(\ref{eq:symms}) implies that the result of the transport lies on a leaf
of the foliation defined by $H_\mu$.

\subsubsection{Construction of transport from the vector fields
$v^{\mu}=(v^{\mu}_1,v^{\mu}_I)$}

Let us recall the construction of the lifting. Choose  a point $a\in
U$ lying on  a leaf $\{H=h\not=0\}$, and  choose a univalued branch
of $H$ equal to $h$ at $a$ defined in some small neighborhood $W$ of
$a$. For a vector $\xi\in T_h\C\cong\C$ denote by $\tilde{\xi}_a$
the only real linear combination of $v^{\mu}_1(a)$ and $v^{\mu}_I(a)$  such that
$dH(\tilde{\xi}_a)=\xi$:
\begin{equation}\label{eq:lifting}
\tilde{\xi}_a=\re\left(h^{-1}\xi\right)
v^{\mu}_1+\im\left(h^{-1}\xi\right)v^{\mu}_I.
\end{equation}

For a germ of a smooth curve $\varrho(t), t\in(-r,r)$ passing
through $h$ and for each point $a'\in H^{-1}(\varrho(t))\cap W$ we
can repeat this construction taking vector $\varrho'(t)$ as $\xi$.
This provides a smooth vector field on real three-dimensional
surface $H^{-1}\left(\varrho((-r,r))\right)$, and the trajectory
$\tilde{\varrho}_a(t)$ of this vector field passing through $a$ is
the required lifting. Evidently,
$H\left(\tilde{\varrho}_a(t)\right)=\varrho(t)$. In other words,
this construction provides a transport of points from one leaf of
the foliation to another along smooth curves in the plane of values $h\in \C$.

It turns out that for $h_0$ sufficiently small any path on the
universal covering of  $\{0<|h|<h_0\}$ can be lifted to $U$ provided
that the starting point $a$ of the lifting lies on the real cycle
$\gamma_{\mu}(h)$ and $|\varrho(t)|'>0$. This allows to transport the real
cycle $\gamma_{\mu}(h)$ to this universal cover: for any
path $\varrho(t)$ in the universal cover  we define the  transport
of $\gamma_{\mu}(h)$ along this path as a union of
liftings of $\varrho(t)$ through all points of
$\gamma_{\mu}(h)$. The result is well defined in a
suitable sense: the continuation depends on the paths chosen, but
continuations along homotopic paths are homotopic (by lifting of
homotopy of the paths). This  provides an analytic continuation of
the pseudo-abelian integral \eqref{eq:integral} to a universal
covering of a punctured disc $\{0<|h|<h_0\}$.

\begin{remark}
The constructed vector fields commute everywhere except in small
neighborhoods of the singular points of the polycycle. In fact, in a
suitable local holomorphic coordinates we have
$v^{\mu}_I=Iv^{\mu}_1$ everywhere, and $v^{\mu}$ defines a holomorphic (in this new complex structure) vector field
everywhere in $U$ except these neighborhoods.
\end{remark}

The rest of the section will be devoted to construction of $v_\mu$.
It will be constructed first in  neighborhoods of singular
points of the polycycle using the local normal forms for the first
integral  near the singular points. Then $v^{\mu}$ will be smoothly
extended to neighborhoods of the arcs of the polycycle joining
them.

We will repeatedly use the following fact, which is an easy consequence of the Cauchy-Riemann equations. Note that multiplication by $i$ on $\bbC^2$ gives rise to the \emph{real} linear endomorphism $J$ on tangent vectors.
\begin{lemma}\label{lem:trivial}
Let $\xi$ be a real tangent vector to $\bbC^2$, $H$ a holomorphic function and $\log H$ its local branch. If $\dr (\log H)(\xi)\in\bbr$ then $\dr (|H|)(\xi)=0$.
\end{lemma}

Also, to simplify notations we will omit the index $\mu$ in $v^\mu$.

\subsection{Construction of $v$ in neighborhoods of saddles}

Let $m_k$ be a saddle of the polycycle $D$.

\begin{lemma} The foliation defined by $H_\mu$ near a saddle point
can be analytically linearized, and the linearization depends
analytically on parameters. Linearizing coordinates $(x,y)$  can be
chosen in such a way that $H=x^{1/\lambda_1} y^{1/\lambda_2}$, where
$\lambda_i$ are analytic functions of $\mu$.
\end{lemma}
This is proved in \cite{BM,N}, and the proof consists of writing the
linearizing coordinates explicitly: if the saddle lies on the
intersection of $\{P_{1,\mu}=0\}$ and $\{P_{2,\mu}=0\}$ then $P_{1,\mu}$ and $P_{2,\mu}$,
multiplied by suitable holomorphic factors invertible near the
saddle, give the linearizing coordinates.

\begin{example} For the form $\theta_{\epsilon, \alpha}$ of (\ref{eq:ea family})
this can be expressed as
$$
H_{\epsilon, \alpha}=x^{a_1}y^{a_2} \quad\text{for}
\quad x=P_1\left(P_3^{a_3}...P_k^{a_k}Q^{\alpha-1/\epsilon}(Q+\epsilon R)^{1/\epsilon}\right)^{1/a_1},  y=P_2.
$$
\end{example}

In the linearizing coordinates the construction of $v$ is easy. Choose some
$0<h<1$.

\begin{lemma}\label{lem:transp s}
For a family of linear saddles $\dot x=\lambda_1 x, \dot y=-\lambda_2 y$ in
a bidisc $\{|x|,|y|\le 1\}$ with the first integral $H=x^{1/\lambda_1}
y^{1/\lambda_2}$ one can construct the pair  of  vector fields
$v=(v_1,v_I)$ defined in $U_s=\{|H|<h<1\}\cap\{|x|,|y|\le 1\}$,
satisfying \eqref{eq:symms} and having the following properties:
\begin{enumerate}
\item both the negative flow of $v_1$ and flow of $v_I$ do not
increase $|x|$ and $|y|$;
\item both $v_1$ and $v_I$ are tangent to
lines $\{y=\const\}$ near $(0,1)$ and to the lines $\{x=\const\}$
near $(1,0)$.\end{enumerate}
\end{lemma}

\begin{proof}
The holomorphic vector field $v_x=\lambda_1 x\partial_x$ preserves
$y$, in particular the transversal $\{y=1\}$, and satisfies
$$
d(\log H)(v_x)=1, \quad d(\log x)(v_x)=\lambda_1>0.
$$
Similarly, the
vector field $v_y= \lambda_2 y\partial_y$ preserves $x$ and the
transversal $\{x=1\}$, and satisfies
$$
d(\log H)(v_y)=1\quad d(\log y)(v_y)=\lambda_2>0.$$

Let $\phi$ be a smooth function defined in $U_s$, $0\le \phi \le 1$,
equal to $0$ in a neighborhood of $\{x=1\}$ and equal to $1$ in a
neighborhood of $\{y=1\}$. We define $v$ as the pair of the real
vector fields $(v_1=\phi v_x+(1-\phi)v_y, v_I=Iv_1)$. One can easily
see that $v$ satisfies conditions of the Lemma.
\end{proof}

Note that $v_1$ (and therefore also $v_I$) are not  analytic vector
fields, as $\phi$ is not analytic.
\begin{proposition}\label{prop:s is stable}
Transport of a real curve $\gamma\subset \{|x|,|y|\le
1\}\cap\{H=h_0\in \R,  h_0<h\}$ along  any smooth curve
$\varrho(t):[0,1]\to \{0<|z|\le|h_0|\}$  remains in $U_s$ if
$|\varrho|'(t)<0$ for all $t$. Moreover, the transport intersects
the transversals $\{x=1\}$ for all $t$ if $\gamma$ intersects it
(and similarly for $\{y=1\}$).
\end{proposition}

This follows from the fact that lifting of $\varrho(t)$ starting
from any point $a\in U_s$ will remain in $U_s$. Indeed,
$|\varrho|'(t)<0$ is equivalent to
$\re\left(\varrho(t)^{-1}\varrho'(t)\right)<0$, so the coefficient
of $v_1$ in (\ref{eq:lifting}) is negative. This implies that $|x|$,
$|y|$ do not increase along the lifting of $\varrho(t)$, due to the
first claim of the previous Lemma.

The second claim follows since $v_1,v_I$ are tangent to both
transversals.

\subsection{Construction of $v$ in neighborhoods of saddle-nodes}

Let $m_k$ be a saddle-node of the polycycle $D$.

\begin{lemma}\label{normall}
There exist two real analytic functions $\epsilon=\epsilon(\mu)$ and $\alpha=\alpha(\mu)$ vanishing at $\mu=0$
and real analytic coordinates $(x,y)$ in some  neighborhood of $m_k$ such that
the vector field
\begin{equation}\label{normal}
\begin{aligned}
\dot x&=-x^2+\epsilon^2,\\ \dot
y&=y\left(1+\alpha(x-\epsilon)\right)
\end{aligned}
\end{equation}
generates the foliation $\theta_{\mu}=0$ in this neighborhood.
The function
\begin{equation}\label{normalfirst}
y(x+\epsilon)^{\alpha}\left(\frac{x-\epsilon}{x+\epsilon}\right)^{1/2\epsilon}
\end{equation}
is a first integral of this vector field.
\end{lemma}

\begin{remark}
Normalizing coordinates for the family (\ref{eq:ea family}) can be given explicitly:
 let $y=P_1P_2^{a_2/a_1}\cdots
P_k^{a_k/a_1}$. Then
$$
H_{\epsilon, \alpha}^{1/a_1}=y(Q/R)^{\alpha/a_1}\left(\frac{Q/R+\epsilon}{Q/R}\right)^{1/{a_1\epsilon}},
$$
which becomes (\ref{normalfirst}) if we take
$X=a_1\left(-Q/R-\epsilon/2\right)$ and rescale $\epsilon$ by
$a_1/2$ and $\alpha$ by $a_1$.
\end{remark}

\begin{remark} It would seem more natural to use as a local model the full
versal deformation of the saddle-node, i.e. the family (\ref{normal}) with
$\epsilon^2$ replaced by $\epsilon$. However, the family of real
polycycles we study extends continuously only to the half of the versal deformation
where singular points resulting from splitting of the saddle-node remain real.
This is the reason for choosing the model (\ref{normal}).

Investigation of another  half of the versal deformation is a separate interesting problem.
\end{remark}

\begin{proof}
The fact that the first integral is preserved by the vector field is
a direct computation.
Existence of normalizing coordinates follows from the general theory of bifurcation of saddle-nodes. Indeed,
from \cite{IY} it follows that (\ref{normal}) is the local \emph{formal} normal form, and it is well-known that for closed forms, due to vanishing of the moduli of analytic classification,  the formal normal form and the analytic orbital normal form coincide.
\end{proof}

Until the end of this section  we will work in the normalizing coordinates and will denote by $H=H_{\epsilon,\alpha}(x,y)$ the first
integral (\ref{normalfirst}) of the model family (\ref{normal}),
\begin{equation}\label{eq:normaldlogH}
\frac{dH}{H}=\frac{dy}y+\frac{d\widetilde{H}}{\widetilde{H}},\quad\text{
where }\frac{d\widetilde{H}}{\widetilde{H}}=
\frac{1+\alpha(x-\epsilon)}{x^2-\epsilon^2}dx.
\end{equation}
In other words,
\begin{equation}
\widetilde{H}(x)=(x+\epsilon)^\alpha\left(\frac{x-\epsilon}
{x+\epsilon}\right)^{\frac{1}{2\epsilon}}
\end{equation}
for $\epsilon\neq0$ and $\widetilde{H}(x)=x^\alpha e^{-1/x}$ for
$\epsilon=0$.

We consider this model in the unitary bidisc $\{|x|\leq1,|y|\leq1\}$.

\begin{lemma}\label{lem:transp sn} For the model family above there exists a pair $v=(v_1,v_I)$ of  vector fields $v$ defined in $\{|x|,|y|<1\}$ (except in a small neighborhood of $(1,1)$) and satisfying \eqref{eq:symms}. Both  the negative flow of $v_1$ and flow of $v_I$ do not increase  $|x|$ and $|\widetilde{H}|$. Both $v_1$ and $v_I$ are tangent to lines $\{y=\const\}$ near $(0,1)$ and to the lines $\{x=\const\}$ near $(1,0)$.
\end{lemma}
\begin{proof}
We consider only the case $\epsilon\neq0$, and the case $\epsilon=0$
is obtained by taking the limit.

Let
\begin{equation}\label{eq:normalautomorphisms}
v_x=\frac{x^2-\epsilon^2}{1+\alpha(x-\epsilon)}\,\partial_x, \qquad
v_y=y\partial_y
\end{equation}
be two vector fields in the bidisc. We have
\begin{align*}
d(\log H)(v_x)=d(\log H)(v_y)=1\\
d(\log \widetilde{H})(v_x)=1,\qquad d(\log \widetilde{H})(v_y)=0\\
d(\log y)(v_x)=0,\qquad  d(\log y)(v_y)=1.
\end{align*}

Let $\phi$ be a smooth function defined in the bidisc, $0\le \phi
\le 1$,equal to $0$ in a neighborhood of $\{x=1\}$ and equal to $1$
in a neighborhood of $\{y=1\}$. One can easily check that the pair
of two real vector field $(v_1=\phi v_x+(1-\phi)v_y, v_I=Iv_1)$
satisfies conditions of the Lemma. \end{proof}

\begin{figure}[htpb]
\centerline{\includegraphics[width=\hsize]{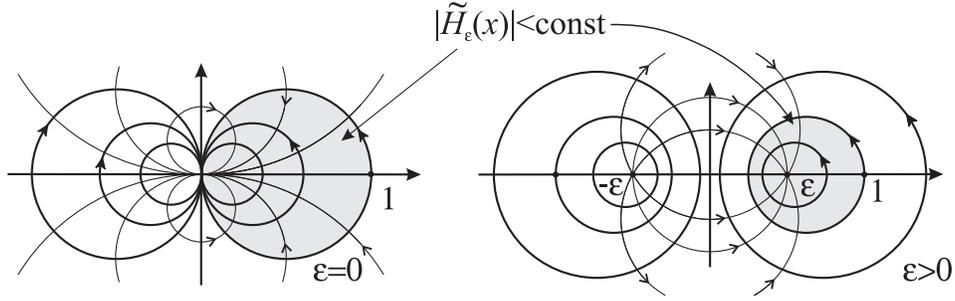}}
\caption{Flow of the real and imaginary parts of the vector fields $v_x$.}\label{fig:cycle}
\end{figure}

The following is a saddle-node analogue of the
Proposition~\ref{prop:s is stable}.
\begin{proposition}\label{prop:homotopy}
Let $\gamma(h_0,\epsilon,\alpha)$ be  a relative cycle in the unitary bidisc
lying on $\{H_{\epsilon,\alpha}(x,y)=h_0\neq0\}$  taken modulo the two
transversals $\{y=1\}$ and $\{ x=1\}$. Assume in addition that the
cycle lies entirely in the bidisc
\begin{equation}\label{eq:smallbidisc}
\left\{|\widetilde{H}(x)|\le|\widetilde{H}(1)| \right\}
\times\{|y|\le 1\}\end{equation}

Then the relative cycle $\gamma(h_0,\epsilon,\alpha)$ transports in relative cycles along  any curve
$\varrho(t):[0,1]\to \{0<|z|\le|h_0|\}$, $\varrho(0)=h_0$,
remains in \eqref{eq:smallbidisc}  provided
$|\varrho|'(t)<0$ for all $t$.
\end{proposition}

Note that unlike the previous case of saddles, the lifting does not
preserve the whole bidisc $\{|x|,|y|\le 1\}$, but only the bidisc
(\ref{eq:smallbidisc}) (see figure \ref{fig:cycle}). However, the  parts of the \emph{real} cycles
$\gamma_{\epsilon,\alpha}(h)$ passing near the saddle-node lie in
(\ref{eq:smallbidisc}), so satisfy the conditions of the Lemma.
\begin{proof}
Indeed, since $v$ preserves the transversals $\{y=1\}$ and $\{
x=1\}$, the endpoints of $\gamma(h,\epsilon,\alpha)$ still lie on
them. Similarly to the proof of Proposition~\ref{prop:s is stable},
from $\re\left(\varrho(t)^{-1}\varrho'(t)\right)<0$ we conclude that both $|\widetilde{H}|$ and $|y|$
only decrease along lifting of $\varrho(t)$, so the points of bidisc
(\ref{eq:smallbidisc}) remain in it when transported along $\varrho(t)$.
\end{proof}

\subsection{Gluing a global transport}
Here we extend the vector fields constructed above to a vector field defined in a whole neighborhood of
the polycycle $D$.

\begin{proposition}\label{prop:glueing v}
There exists a complex neighborhood $U$ of $D_{\mu}$ and
a pair of real vector fields $v=(v_1,v_I)$ in $U$ satisfying
\eqref{eq:symms}. Moreover, transport of real cycles
$\gamma_{\mu}(h)$ along  any curve
$\varrho(t)\subset\{0<|z|\le|h_0|\}$  remains in $U$  provided
$|\varrho|'(t)<0$ for all $t$.
\end{proposition}

\begin{proof}
For each singular points of the polycycle we defined two
transversals intersecting the polycycle. They are given by $\{x=1\}$ and
$\{y=1\}$ in the normalizing chart of the singular point.

For an arc of the polycycle joining two singular points $m_1,m_2$
consider two transversals $\Gamma_1,
\Gamma_2$ to this arc, lying in normalizing charts $W_1$ and $W_2$ of $m_1$ and
$m_2$ correspondingly, and let $K$ be a compact piece of the arc joining $\Gamma_1$ and $\Gamma_2$. Let $U_K$ be a neighborhood of $K$ in the
leaf of the foliation containing $K$.

To fix notations, assume that in the normalizing coordinates in $W_1$  the leaf
containing $K$ is contained in $\{x=0\}$. The family $\mathcal{F}_1$
of discs given by  $\{y=const\}$ is transversal to $U_K$ and
invariant under the flow of vector fields $v_1,v_I$ constructed
before (assuming $U_K$ is sufficiently small). Similar transversal
family $\mathcal{F}_2$ of invariant discs exists on the other end of
$K$.

Our immediate goal is to embed these two families of discs into
one smooth  family $\mathcal{F}$ of smooth real two-dimensional
discs transversal to $U_K$ and filling some neighborhood of $K$ in
$\C^2$. Let $g_1$ be a Riemannian metric defined in $W_1$ which in
normalizing coordinates is just a standard Euclidean metric in
$\R^4$, so the leaf containing $K$ and the discs of $\mathcal{F}_1$
lie in orthogonal affine planes. Let $g_2 $ be a similar metric in
$W_2$, and continue smoothly these two metrics to a metric $g$
defined in some neighborhood of $U_K$ in $\C^4$. We can assume that $g$ preserves the
complex structure of $U_K$. The
exponential mapping $\exp_g$ maps diffeomorphically
some neighborhood $\widetilde{W}\subset NU_K$ of $U_K$  in its
normal bundle $NU_K$ onto some neighborhood $W$ of
$U_K$ in $\C^2$,  in such a way that the images of fibers $N_xU_K$
are mapped into the leaves of $\mathcal{F}_i$ for $x\in U_K\cap
W_i$, $i=1,2$. We define $\mathcal{F}$ as the family
$\mathcal{F}(x)=\exp_g(B_x)$, $x\in U_K$, where
$B_x= N_xU_K\cap\widetilde{W}$ are small discs
(symmetry of $g$ with respect to conjugation assures that for
$x\in K$ the leaves $\mathcal{F}(x)$ intersect $\R^2$ by a smooth curve transversal to $K$).

Shrinking $B_x$, we can assume that $\mathcal{F}(x)$ is transversal
to the leaves $\{H=h\}$ for all sufficiently small $h$ (because
$\mathcal{F}(x)$ is transversal to the leaf containing $K$). We
define $v=(v_1, v_I)$ in the neighborhood $W_K$ of $U_K$ in $\C^2$
as the two vector fields tangent to $\mathcal{F}(x)$ and satisfying
\eqref{eq:symms}.

Evidently, $\mathcal{F}$ coincide with $\mathcal{F}_i$ in $W_i$.
Since \eqref{eq:symms} define uniquely the pair of vector fields
tangent to a real two-dimensional surface transversal to $\{H=h\}$,
we conclude that thus constructed $v$ is a smooth extension of the
vector fields constructed before.

Dynamics of $v$ on each leaf $\mathcal{F}(x)$ is conjugated to
dynamics on the transversal $\{y=1\}$ of $v$ constructed in either
Lemma~\ref{lem:transp s} (when one of two singular  points is a saddle)
or Lemma~\ref{lem:transp sn} (for a connection between two saddle-nodes),
with conjugation map being just the flow from one transversal to another.
We use here the fact of smoothness of $Q$: it implies that the weak manifolds of saddle-nodes
of the polycycle $D$ join them to saddles, and two saddle-nodes can be connected by their strong manifolds
 only.

Let $a$ be a real point lying on $\mathcal{F}(x)\cap\gamma_{\mu}(h)$
for $h$ sufficiently small. Then the  lifting of any curve
$\varrho(t)\subset\{0<|z|\le |h|\}$, $\varrho(0)=h$,  starting
from $a$ remains in $\{|H|\le h\}\cap \mathcal{F}(x)$, which is
contained in $W_K$ provided that $h$ is sufficiently small.

Repeating the construction for all arcs of the polycycle, we get the
pair $v=(v_1, v_I)$ defined in the neighborhood $U$ of the
polycycle, where $U$ is the union of normalizing charts $W_i$ and the
neighborhoods $W_K$ of all singular points and arcs of the
polycycle.\end{proof}

\section{Pushing cycles away from the weak manifold}
Recall that $L_E^\R$ is the union of edges $D$ contained in the zero level curve $Q=0$.
Any  arc of $D$ lying on $\{Q=0\}$ joins two saddle-nodes, and
is the strong manifold of both.

The aim of this section is to prove the following Proposition:

\begin{proposition}\label{prop:vvar}
There exists a neighborhood $U(L_E^\R)$ of $L_E^\R$  in $\C^2$ and
neighborhoods $U_k^\C$ of the central varieties of saddle-nodes
$m_l\in D$ such that in the open set $E_L=U(L_E^\R)\setminus\cup\overline{ U_l^\C}$
\begin{enumerate}
\item the family
$\theta_{\mu}$ defines a holomorphic foliation without singularities analytically
depending on $\mu$ and
\item the family of cycles
\begin{equation}\label{vvar}
Var_{a_1,...,a_k}\gamma_\mu(h).
\end{equation}
is homotopic along the fibers to the family of cycles lying in $E_L$, where $\gamma_\mu(h)$ are real cycles as in \eqref{eq:integral}.
\end{enumerate}
\end{proposition}

Shrinking $E_L$ if necessary, we can assume that connected components of $E_L$ are in one-to-one correspondence of the arcs of $D$ lying on $\{Q=0\}$ and have homotopy type of the figure eight.

We first show that a cycle lying near $L_E$ and in the saddle regions of the saddle-nodes
can be pushed away from the central variety while remaining in a neighborhood of $L_E$.
This will be needed to prove that the integral of a meromorphic form $M^{-1}\eta_{\mu}$ over such cycle depends holomorphically on $\mu$ and the transversal coordinate. The transversal coordinate is
exactly $\frac{1}{\omega(h,\epsilon(\mu),\alpha(\mu))}$ for suitable functions $\epsilon(\mu),\alpha(\mu)$.

\begin{lemma}\label{lem:normalpushing}
Using assumptions and notations of Proposition~\ref{prop:homotopy} let $\gamma=\gamma(h,\epsilon,\alpha)\subset \{H_{\epsilon,\alpha}=h\}$ be a relative cycle
lying in the bidisc (\ref{eq:smallbidisc}) and whose boundary is in
$\{H_{\epsilon,\alpha}=h\}\cap\{y=1\}$. Then $\gamma$ is homotopy equivalent in $\{H_{\epsilon,\alpha}=h\}$ to a
relative cycle $\tilde \gamma$ with the same property and, in
addition, not intersecting a neighborhood $\{|y|<\delta\}$ of the
$x$-axis, for a sufficiently small  $\delta>0$ independent of
the cycle.
\end{lemma}

\begin{proof}[Proof of Lemma \ref{lem:normalpushing}.]
Choose a non-negative bump function $\psi(y)$ equal identically to
$1$ on  $\{|y|<\delta\}$, and vanishing outside $\{|y|<2\delta<1\}$.
Define the vector field  $V=\psi (v_y-v_x)$, where $v_x$ and $v_y$
were defined in (\ref{eq:normalautomorphisms}). Evidently, $dH(V)=0$,
so the flow of $V$ preserves the foliation. We can assume that
$c=\dist(\gamma,\{y=0\})<\delta$. Consider the image
$\tau_M\gamma$ of the cycle $\gamma$ by the $M$-time flow, where
$M=\log\frac{\delta}{c}>0$. Since $L_Vy=y$ in $\{|y|\le\delta\}$,
the image $\tau_M\gamma$ lies outside $\{|y|\le ce^M=\delta\}$.
Since $L_V(\log\widetilde{H})=-\psi<0$, the $|\widetilde{H}|$ is
decreased by this flow, so the condition (\ref{eq:smallbidisc}) is
still satisfied.
\end{proof}

\subsection{Flow-box triviality}

Consider a  neighborhood $U(L_E^\R)$ of $L_E^\R$  in $\C^2$ which is
a union of the normalizing charts of the saddle-nodes and of the open set $U_K$ constructed in the proof of Proposition~\ref{prop:glueing v} for $L_E^\R$.
Let $U_k^\C$ be neighborhoods of the central variety of each
saddle-node $m_k$ as in Lemma~\ref{lem:normalpushing}.

\begin{lemma}\label{lem:boldeight}
The foliation defined by $H_{\mu}$ in the open set
$E_L$ is analytic without
singularities and depends analytically on sufficiently small
parameter $\mu$.
\end{lemma}
\begin{proof}
Indeed, by construction $E_L$ is covered by several charts, namely
neighborhoods of bifurcating saddle-nodes and neighborhoods of
compact subsets of separatrices on some positive distance from the
saddlenodes. In each of these sets the foliation defined by
$H_{\epsilon}$ can be brought analytically to a suitable normal
form, either to normal form of Lemma~\ref{normall} or just to the
standard flow box. Evidently, $E_L$
lies on a finite distance from singularities. \end{proof}

\begin{proof}[Proof of Proposition~\ref{prop:vvar}.] The cycle $\gamma$ can be continuously
moved to close leafs of the foliation by Proposition~\ref{prop:glueing v}.
It was proved in \cite{BM} that the pieces of $\gamma$ lying near
saddles or near separatrices lying on $P_i=0$ are annihilated by the
operator $Var_{a_1...a_k}$. Therefore the cycle
$Var_{a_1...a_n}\gamma$ is supported in
$U(L_E^\R)$. Moreover, it still lies in  (\ref{eq:smallbidisc})
in normal coordinates, so by Lemma~\ref{lem:normalpushing} it is
homotopically equivalent along the fibers to a cycle $\gamma'(h)$ lying in $E_L$.
\end{proof}

\section{Proof of Theorem \ref{var}}

Let $z$ be a holomorphic coordinate on a transversal $\Gamma$ to
$\{Q=0\}$.

\begin{lemma}\label{lem:trasvparam}
For the family $\theta_{\mu}$ the coordinate $z$ is a holomorphic  function of $-\frac
{1}{\omega(H_{\Gamma},\epsilon(\mu),\alpha(\mu))}$, where $\epsilon(\mu), \alpha(\mu)$
are some analytic functions of $\mu$ which are the same for any two transversals to the same arc of $D$.
\end{lemma}

\begin{remark}
Functions $\epsilon(\mu),\alpha(\mu)$ from lemma \ref{normall} and lemma \ref{lem:trasvparam} coincide.
\end{remark}

\begin{proof} Every transversal can be holomorphically mapped to a transversal lying in a normalizing chart of some
saddle-node of the polycycle $D$, just by the flow of the vector field tangent to the foliation.
Therefore, the claim follows from Lemma~\ref{normall}: when restricted to $\{y=1\}$, the first integral \eqref{normalfirst} becomes \eqref{omega}, up to a linear change of $\epsilon,\alpha$.\end{proof}

\begin{remark} The parameters $\epsilon(\mu),\alpha(\mu)$ are intrinsically
defined: $1/\epsilon(\mu)$ is the residue, and $\alpha(\mu)$ is the sum of
residues of the restriction to $\Gamma$ of the form $\theta_{\mu}$. For the family (\ref{eq:ea family})
the smooth irreducible
double divisor $\{Q=0\}$ is split into two close irreducible smooth
curves $\{Q=0\}$ and $\{Q+\epsilon R=0\}$, with residues $1/\epsilon$ and $\alpha - 1/\epsilon$ being the same for all transversals. In general, the residues are locally constant along $\{Q=0\}$ (e.g. by closedness of $\theta_{\mu}$), but can be different for different connected components. \end{remark}

\begin{lemma}\label{lemaomega}
For sufficiently small $\epsilon$ the mapping  $h\mapsto
-\frac{1}{\omega(h,\epsilon,\alpha)}$ is one to one on the interval
$[0,1]$. \qed
\end{lemma}

Let $B_{\mu}$ be some small polydisc, and consider a
foliation $\mathcal{F}$ of $E_L\times B_{\mu}$ by
one-dimensional leaves
$\{H_{\mu}=h,\mu=\const\}.$ According
to Lemma~\ref{lem:boldeight} this is an analytic foliation without
singularities.

\begin{lemma}\label{lem:AIparam}
Let $\gamma$ be a closed connected curve on a leaf of $\mathcal{F}$
and assume that it can be continuously transported to nearby leafs.
Denote the resulting family by $\gamma_{\mu}(z)$, where
$z$ is the coordinate of  a point of the intersection of the cycle
and some fixed transversal to $\{Q=0\}$. Let $\eta_\mu$ be a meromorphic
one-form in $E_L\times B_{\mu}$ such that $Q_{1,\mu}Q_{2,\mu}\eta_\mu$ is holomorphic. Then there exist two analytic functions
$\epsilon(\mu),\alpha(\mu)$ such that the
integral $I_{\mu}(z)=\int_{\gamma_{\mu}(z)}\eta_\mu$ is a meromorphic
function of $z$ and depends
analytically on $\mu$.
\end{lemma}
\begin{proof}
A connected component of the open set $E_L\times B_{\epsilon,\alpha}$ containing $\gamma_{\mu}(z)$ is covered by two normalizing charts of neighborhoods of saddle-nodes (with a
neighborhood of weak manifold removed)  and a neighborhood of the
connection between saddle-nodes. In each of the above charts leafs
of our foliation are graphs of (multivalued) functions $x(y,h)$ of
the coordinate $y$ along the leaf $\{Q=0\}$. Therefore in each chart
the curve $\gamma$ can be written as a curve
$(x(t),y(t),\mu)$, and we can define its projection
curve $(0,y(t),0)$ lying on $\{Q=\mu=0\}$. It is important here that by Proposition \ref{prop:vvar} we can keep the cycle away from the weak manifold where the projection is not regular.

We can join $\gamma$ and its projection by a continuous family of
closed curves lying on leaves of foliation using the explicit normalizing charts.
We can do it in each normalizing chart, and the condition of trivial  holonomy of $\gamma$ guarantees that these pieces will glue together. This implies that the holonomy of the
projection curve is trivial, so $\gamma$ can be continued from $L$ to all
nearby leaves. Therefore $I(z)$ is univalued in a neighborhood of
$z=\mu=0$. Since the length of the continuation is
bounded, the growth of $I(h)$ is at most polynomial.
\end{proof}

\begin{lemma}\label{lem:logrotexp}
Define the functions $g_\beta(z,\epsilon,\alpha)$ by
$$g_\beta\left(-\frac{1}{\omega(he^{i\beta},\epsilon,\alpha)},\epsilon,\alpha\right)=
-\frac{1}{\omega(h,\epsilon,\alpha)}.
$$
Then for any $\beta_0>0$ and any neighborhood $W\subset\C$ of the
origin there exists a small tridisc
$W'\subset\C^3_{z,\epsilon,\alpha}$ near the origin such that the
function $g_{\beta'}(z,\epsilon,\alpha)$ maps
$W'\times\{|\beta|<\beta_0\}$ holomorphically into $W$.\end{lemma}

\begin{proof}
The function $g_\beta(z,\epsilon,\alpha)$ is the $i\beta$-time flow
of the vector field $\widetilde{v'}=\frac{z^2+\epsilon z}{1+\alpha
z}\partial_z$, which is just the vector field  $v_x$ of
(\ref{eq:normalautomorphisms}) up to an affine change of variables. Therefore the
claim follows from the fact that $x=0$ is a fixed point of
$\widetilde{v'}$ for $\epsilon=\alpha=0$ and analytic dependence of
the solution of ODE on the initial conditions and
parameters.\end{proof}

\begin{proof}[Proof of Theorem \ref{var}.]

By Proposition~\ref{vvar} and the definition of the operator
$Var_{a_i}$ the cycle
$\gamma'=Var_{a_1,...,a_n}\gamma(h,\mu)$ is a union of
several disjoint cycles $\gamma'_i$ lying in $E_L$ on leaves
$\{H=he^{i\beta_i}\}$, for finitely many $\beta_i\in \R$. Since
$\gamma'$ can be continued by $h$, the cycles $\gamma'_i$ also can
be continued by $h$. Therefore by Lemma~\ref{lem:AIparam} and Lemma~\ref{lem:trasvparam} the
function $Var_{a_1,...,a_n}I_{\mu}$ is a finite sum of
$f_i(-\frac{1}{\omega(he^{i\beta_i},\epsilon_i(\mu),\alpha_i(\mu))},\mu)$,
where each $f_i$ is holomorphic in some bidisc at the origin. Then
Lemma~\ref{lem:logrotexp} implies it is an analytic  function of
$-\frac{1}{\omega(h,\epsilon,\alpha)}$.

\end{proof}

\section{Proof of Theorem~\ref{main:localan}}
\begin{proposition}\label{mainlemma}
Application of the operator $Var_{a}$ decreases the number of
zeros of $I_{\mu}(h)$ by at most some finite number uniformly bounded from
above and depending on the family $\Theta$ only.
\end{proposition}
\begin{proof}
To prove the Proposition, consider the sector $\{r<|h|<1,
|\arg h|\le \alpha\pi\}$. Proposition~\ref{pr:zersing} guarantees
that the zeros of $I(h)$ do not accumulate to $0$, so for $r$
small enough this sector includes all zeros of $I(h)$ on $(0,1)$. To
count the number of zeros of $I(h)$ in this sector apply the
argument principle. As in \cite{BM,N}, the increment of argument of
$I(h)$ on the counterclockwise arc $\{|h|=1, |\arg h|\le
\alpha\pi\}$ passed counterclockwise is uniformly bounded from above
by Gabrielov's theorem \cite{gabrielov}. Here we need the analytic dependence of the
compensator function $\omega(u,\epsilon, \alpha)$ on
the parameters $\epsilon, \alpha$, when $|u|=const$. This is proved in Proposition
\ref{ancomp}.

Proposition~\ref{pr:zersing} below implies that the increment of argument along the small arc $\{|h|=r, |\arg h|\le \alpha\pi\}$ passed clockwise is uniformly bounded from above as well. The classical Petrov's argument now shows that the increment of argument of $I(h)$ along the segments $\{r<|h|<1, |\arg h|=\pm\alpha\pi\}$ is bounded from above by the number of zeros of $Var_{\alpha}I(h)$, which  proves the Proposition.\\
\end{proof}

\begin{proof}[End of the proof of Theorem~\ref{main:localan}]
Theorem~\ref{main:localan} follows from Proposition~\ref{mainlemma},  Theorem~\ref{var}
and the fact that the number of zeros of
$$
f=\sum_{i} f_{i}(-\frac{1}{\omega(h,\epsilon_i(\mu),\alpha_i(\mu))},\epsilon_i(\mu),\alpha_i(\mu),\mu),
$$
i.e. of the right-hand side of (\ref{eq:thmvar}), on some interval $(0,r)$
is uniformly bounded for all sufficiently small $\mu$. The latter claim is a direct application of fewnomials theory of  Khovanskii \cite{kh}: since all $-\frac{1}{\omega(h,\epsilon,\alpha)}$ are Pfaffian functions, see (\ref{eq:omega is Pfaffian}), the upper bound for this number of zeros can be given, using Rolle-Khovanskii arguments of \cite{fewnomials}, in terms of the
number of zeros of some polynomials in $F_i$ and their derivatives. The latter are uniformly bounded by Gabrielov's theorem \cite{gabrielov}.\end{proof}

The aim of the following Proposition \ref{pr:zersing} is to describe
the asymptotics of  the pseudo-abelian integral $I(h)$ and its
variation at $h=0$. This justifies the application of
Petrov's argument in the proof of Theorem~\ref{main}.

The regular form of the singularity together with an a priori bound
for the growth of the integral $I(h)$ gives us an estimate for the
increment of the argument along arcs of small circles around $h=0$.
Note that the singularity at $\ep\neq0$ case was already
investigated \cite{BM,N}. Thus, it remains to investigate the
non-trivial exponential case $\ep=0$.

\begin{proposition}
\label{pr:zersing} Let $I(h)$ be a non-zero multi-valued holomorphic
function on a neighborhood of $h=0$ verifying the iterated variation
relation \eqref{eq:thmvar} for some $k$ and satisfying the a priori
bound
\begin{equation}
\label{mdgrowth} |I|\leq C |h|^{-N}
\end{equation}
in sectors $\{|\arg h|\le A\}$.

Then $I(h)$ has a leading term of the form $h^{\alpha}(\log h)^k$ or
of the form $(\log h)^{-k}(\log (\log h))^l$, with $k,l>0$.
Moreover, for any $N^\prime > N$ the increment of argument of $I(h)$
along the arc $C_0=\{r e^{i\phi}\vert \phi\in[-A,A]\}$ traveled
clockwise can be estimated from above
\begin{equation}
\label{zerest} \Delta \Arg_{C_0} I \leq 2N^\prime A,
\end{equation}
for all sufficiently small $r>0$.
\end{proposition}

\section{Generalized Roussarie-Ecalle compensator}
\label{sec:apn1}

In this section we prove the existence of the generalized
Roussarie-Ecalle compensator \eqref{omega}. We start with the
following, general statement

\begin{lemma}
\label{lem:inverse} Let $r(x)$ be a rational function. There exist a
holomorphic, multivalued, endlessly continuable function $\omega(z)$
which satisfies the following equation
\begin{equation}
\label{inveq} \omega\prime (z) = r(\omega(z)).
\end{equation}
The ramification set of the function $\omega(z)$ is discrete along
any path.
\end{lemma}

\begin{proof}
Consider the Riemann sphere $\overline{\bbC}$ with small disjoint,
open discs $D_1,\ldots,D_k$ centered at zeroes and poles of $r(x)$.
Let the initial condition $x_0\in \overline{\bbC}$ be chosen away
from these discs. Let $z=l(s), \ s\in [0,1]$, $l(0)=0$ be a path in
$\bbC$ starting at $z=0$. Since the domain $\overline{\bbC}\setminus
(\sqcup_j D_j)$ is compact, the solution of the equation
$\omega\prime = r(\omega)$ is well defined along $l$ at least until
it enters to some disc $D_j$., i.e. for $s\in[0,s_j]$. The solution
can be extended to a holomorphic function in a neighborhood of this
segment of $l$.

In a disc $D_j$ there exists a holomorphic coordinate $\xi$ such
that the equation takes the following (normal) form
\begin{equation*}
\label{zpeq} \xi\prime = \begin{cases}
a \xi^{n},\quad  a\in\bbC^* & \text{for } n\leq -1\\
a \xi,\quad  a\in\bbC^* & \\
\left( r \xi^{-1}+ a \xi^{-n} \right)^{-1},\quad  a\in\bbC^*,\
r\in\bbC & \text{for } n\geq 2.
\end{cases}
\end{equation*}
The solution reads respectively
\[
t-t_0 = \begin{cases}
a^{-1}(1-n)^{-1} \xi^{1-n} \\
a^{-1} \log \xi  \\
 r \log \xi+ \tfrac{a}{1-n} \xi^{1-n}
\end{cases}
\]
Now, if $n\geq 1$, the solution $\omega$ can not reach the singular point $\xi = 0$, so it either leaves the disc $D_j$ or stays inside (and is well defined) for $s\in[s_j,1]$. If $n\leq -1$, then the singular point $\xi=0$ corresponds to the ramification of the solution $\omega$.\\
\end{proof}

Now we return to the particular problem related to the existence of
the compensator. One checks that the compensator function $\omega$ in
the logarithmic coordinate $u=\log h$ must satisfy the following
differential equation
\begin{equation}
\label{compeq}
\omega\prime (u) = \frac{\omega
(\omega-\ep)}{1+\alpha(\omega-\ep)}.
\end{equation}
Thus, by Lemma \ref{lem:inverse}, \emph{for fixed} $\ep$ the
solution is a well defined multivalued holomorphic function. The
dependence on $\ep$ is not automatically analytic since in the
equation \eqref{compeq} the collision of two zeroes (at $\omega=0$)
and the collision of zero and pole (at $\omega=\infty$) occur for
$\ep=0$. We overcome these difficulties by taking respective
blow-ups. More precisely, the following proposition holds. Recall
that the lagaritmic chart $u=\log h$ assumed.
\begin{proposition}
\label{ancomp} There exists a positive constant $l_0$ and three
functions $F_S(\ep,s)\ F_E(\ep,u),\ F_N(\ep,w)$ analytic in $\ep$,
analytic multivalued in $s,u,w$ respectively such that in a
neighborhood of any $u_0$ the compensator $\omega(\ep;u)$ has one of
the following forms (depending on the value $\omega(\ep;u_0)$)
\begin{equation}
\label{compsol}
\omega(e^u,\ep,\alpha) = \begin{cases}
\ep\, F_S(\ep,\ep\,(u-u_0))\\
F_E(\ep,u-u_0)\\
\alpha^{-1}\, 1/ F_N(\ep,\alpha^{-1} (u-u_0))
\end{cases}.
\end{equation}
Moreover, these expressions are valid for all paths starting at
$u_0$, of length bounded by $l_0$.
\end{proposition}
\begin{remark}
The indices $S,E,N$ of functions come from the south pole, equator
and north pole on the Riemann sphere.
\end{remark}

\begin{proof}
In the whole proof the logaritmic chart is assumed $u=\log h$. We
will use the notation $\omega(u,\ep)$.
One easily observes that the equation \eqref{compeq} has the
following singularities: zeros of order $1$ at $\omega=0$,
$\omega=\ep$ and $\omega=\infty$ and pole of order 1 at
$\omega=-\alpha^{-1}+\ep$. For $\ep=0$ they degenerate to a single
pole of order $2$ at $\omega=0$. Let two discs centered at $0$ and
$\infty$ respectively, both of radius $r_0/2$ contain all these
singularities for $|\ep|<\ep_0$. Thus, on the ring
$R=R(r_0,r_0^{-1})$ the rational function $\tfrac{\omega
(\omega-\ep)}{1+\alpha(\omega-\ep)}$ is bounded by a constant $M$.
Let $\omega(u_0,\ep)=\omega_0\in R$ and
$\mathrm{dist}(\omega_0,\partial R)=\delta$. Analytic continuation
of $\omega$ along any path $l$ starting at $t_1$, of length $\leq
\delta/M$ is so contained in $R$ and satisfies estimate $|\omega -
\omega_0|\leq M\, |l|$. Moreover, this solution depends analytically
on $\ep$. Defining the "base" solution $F_E$ on the ring $R$ by the
initial condition $F_E(\ep,0)=1$ we get
\[
\omega(u,\ep) = F_E(\ep,u-u_0),
\]
where $u_0=u_1 - \int_1^{\omega_0}
\frac{1+\alpha(\omega-\ep)}{\omega (\omega-\ep)} \dr \omega$.

Now, we consider the lower semi-sphere $|\omega| < 1$ in the Riemann
sphere $\overline \bbC$. We make the following blow up
transformation
\[
\omega= \ep\, y, \qquad s=\ep\, u.
\]
The equation \eqref{compeq} takes the form
\[
y\prime = \frac{y (y-1)}{1+\ep \alpha(y-1)}.
\]
The solution $y=F_S(\ep,s)$, fixed by the initial condition
$F_S(\ep,0)=\tfrac 12\, \ep^{-1}$, is $\ep$-analytic as far as it
remains in a safe distance from "upper" singularities, e.g.\ if
$|y|<2/\ep$. Thus, the compensator reads $\omega(u,\ep) = \ep
F_S(\ep,\ep(u-u_0))$ and this formula is valid along any path of
length bounded by $1/M$, provided $|\omega_0|<1$.

Finally, on the upper semi-sphere $|\omega| > 1$, the blow up map
$x=\alpha^{-1} /z,\ s= \alpha^{-1}\, u$ transforms the equation
\eqref{compeq} to
\[
z\prime = - \frac{z (1-\alpha\ep\, z)}{1+z -\alpha \ep\, z}.
\]
We fix the solution $F_N(\ep,s)$ which is $\ep$-analytic in the region $|\omega_0|>1/2$. Thus, the following formula for compensator $ $ remains valid along any path of length bounded by $1/M$, provided $|\omega_0|>1$.\\
\end{proof}

\section{Proof of Proposition \ref{pr:zersing}}

Note that it is enough to proof the statement pointwise with respect
to all parameters, in particular $\ep$. As the case $\ep\neq0$ was
already investigated \cite{BM}, it remains to prove the claim in the
non-trivial exponential case $\ep=0$.

The general strategy of the proof is the following. We construct
explicitly  a particular solution of the variation equation \eqref{eq:thmvar}.
Since solutions of the corresponding homogeneous variation equation (i.e. $\var_{a_1,\ldots,a_k I\equiv 0}$) were
already considered in \cite{BM}, this gives us a description of the
general solution. To construct a particular solution of
\eqref{eq:thmvar} we first solve it explicitely up to a sufficiently
small remainder on the right-hand side (Lemma \ref{lem:prf}). Next
the solution to the new equation is found in terms of convergent
series (Lemma \ref{lem:sums}).
\begin{remark}
This stategy is in the spirit of the two steps construction of a
solution of the homological equation associated to the normal form
problem for diffeomorphisms and vector fields \cite{r1,IY}.
\end{remark}

In this section we will work in the logarithmic chart $u=\log h$. In
this coordinate the variation operator $\var_a$ \eqref{vardef}
becomes  a difference operator
\begin{equation}
\label{diffdef} \diff_a f=f(u+ia\pi)-f(u-ia\pi).
\end{equation}
We introduce also the notation for the iterated differences
\begin{equation*}
\diff_{a_1,\ldots,a_k}:= \diff_{a_1}\cdots \diff_{a_k}.
\end{equation*}
The multivalued functions defined in a punctured neighborhood of
$h=0$ become functions holomorphic in the half-planes
$\hp_{L-}=\{\re u < -L\ll 0\}$. All functions below are assumed to
be of this type.


Let $\mathcal{P}(u)$ be the  space $\C[u,\tfrac1u,\log u]$ of
polynomials in $\log u$ and Laurent polynomials in $u$.
\begin{lemma}
\label{lem:prf}
Assume that $f(\tfrac1u,\tfrac{\log u}{u})$ is a
holomorphic function of the second variable $\tfrac{\log u}{u}$ and
meromorphic function of $\tfrac{1}{u}$.

\begin{enumerate}
\item For any real $A\in\bbr$ there exists a polynomial $p\in\mathcal{P}(u)$ such that
\begin{equation}
\label{ppr} |(f- p)|\leq M |u|^{-A}
\end{equation}
for some constant $M$.
\item The space $\mathcal{P}$ is closed under the integration operation, i.e.\ for any $p\in \mathcal{P}$ there exists $P\in\mathcal{P}(u)$ such that $P^\prime =p$.
\item For any real $A\in\bbr$ there exists a function $P_f\in \mathcal{P}$ such that
\begin{equation}
\label{prvar} |(f-\diff_{a_1,\ldots,a_k} P_f)|\leq M |u|^{-A}
\end{equation}
for some constant $M$.
\end{enumerate}
\end{lemma}
\begin{proof}
{\it (1)} The function $f$ has the following power series expansion
\[
f= \sum_{m\geq 0,\; l\geq -l_0} a_{ml} \frac{\log ^m u}{u^{m+l}}
\]
We define $p$ to be the sum of all terms with $m+l\leq A+1$; this sum is
finite, so $p\in\mathcal{P}(u)$.

{\it (2)} We use the induction by $(\log u)$-degree of $p$. If $p$ is
a Laurent polynomial in $u$, the integral $\int p$ is a sum of a
Laurent polynomial in $u$ and a term $a\, \log u$, $a\in \bbC$. Consider relations
\begin{equation}
\label{monder} (u^l\, \log^m u)^\prime = l u^{l-1}\log^m u + m
u^{l-1}\, \log^{m-1} u,\qquad (\log^m u)^\prime = m u^{-1}\,
\log^{m-1}.
\end{equation}
Let $p\in \mathcal{P}(u)$ be an element of $\log u$-degree $\leq m$.
The integral $\int p$ is a sum of terms of $(\log u)$-degree $\leq m$
and $a\, \log^{m+1}u$, $a\in\bbc$.

{\it (3)} Points {\it (1), (2)} and simple induction reduce problem
to the following observation. For any $p\in\mathcal{P}(u)$ the
leading term of the solution to the difference equation $\diff_a
F=p$ is given by the integral $P=\int p$, i.e.
\[
|p|\leq M |u|^{-A} \Rightarrow \ |p-\diff_a \tfrac{1}{2\pi i a}
P|\leq \tilde{M} |u|^{-(A+1)}.
\]
We estimate
\begin{multline*}
|p-\diff_a \tfrac{1}{2\pi i a} P|= |p(u) - \tfrac{1}{2\pi i
a}\int_{u-\pi ia}^{u+\pi ia} p(s)|=|\tfrac{1}{2\pi i a}\int_{u-\pi
ia}^{u+\pi ia}(p(s)-p(u))| =\\
 = |\tfrac{1}{2\pi i a}\int_{u-\pi
ia}^{u+\pi ia}p^\prime (u+\xi_s)|\leq M |u^{-(A+1)}|.
\end{multline*}
The last inequality follows from the estimate $|p^\prime|\leq M^\prime |u|^{-(A+1)}$ valid for arbitrary $p\in\mathcal{P}(u)$ satisfying $|p|\leq M |u|^{-A}$.\\
\end{proof}


Let $Q_+$ (resp. $Q_-$) be an upper-left (resp. lower-left) quarter-plane defined as follows
$Q_+=\{u\in\bbc: \re u < -L,\ \im u > - K\}$ and $Q_-=\{u\in\bbc: \re u < -L,\ \im u <
K\}$ for some positive constants $K,L$. We construct here a solution
of the variation equation in $Q_+$. This is sufficient for our
purposes, since for application of the Petrov's argument we need
only estimates in a half-strip $\{\re u<-L, |\im u|< K\}$ with some
finite $L,K>0$.
\begin{lemma}
\label{lem:sums} Let $f$ be a holomorphic function on $Q_\pm$ which
satisfies the estimate $|f (u)|\leq M|u|^{-A}$ on $Q_\pm$ for some
constant $M$. Assume that $A>n$. Then the following series
\begin{equation}
\label{sumsf} F_\pm= (\mp 1)^k \sum_{m_1,\ldots,m_k>0} f\Big(u\pm
2\pi i(a_1\,m_1+\cdots+a_k\,m_k) \mp \pi i (a_1+\cdots+a_k) \Big)
\end{equation}
converges and solves the difference equation on $Q_\pm$
\begin{equation}
\label{diffeqs} \diff_{a_1,\ldots,a_k} F_\pm = f, \qquad a_j>0.
\end{equation}

Moreover, the solution $F_\pm$ is of order $A-n-\ep$, i.e.\ for all
$B<A-n$ the solution $F_\pm$ satisfies the estimate
\begin{equation}
\label{sumsest} |F_\pm|\leq M_{A-n-B}|u|^{-B}.
\end{equation}
\end{lemma}
\begin{proof}
By induction, it is enough to prove the following statement. Let
$|f|\leq M |u|^{-A}$, $A>1$ on $Q_\pm$. Then the formula
\begin{equation}
\label{stepsumsf} F_\pm = \mp \sum_{m=1}^\infty f\Big(u\pm (2\pi
ia\, m-\pi i a)\Big)
\end{equation}
solves the difference equation $\diff_a F_\pm = f$ and $F_\pm$
satisfies the estimate
\begin{equation}
\label{stepest} |F_\pm|\leq M_{A-1-B}|u|^{-B}
\end{equation}
for $B<A-1$.

The series \eqref{stepsumsf} is convergent, so the function $F_\pm$
is well defined. A direct computation shows that it satisfies the
difference equation. We estimate
\begin{multline*}
|F_\pm|\leq \sum_m |f\Big(u\pm (2\pi ia\, m-\pi i a)\Big)| \leq \\
M |u|^{-B} \sum_m \Big|\tfrac{u}{u\pm (2\pi ia\, m-\pi i a)}\Big|^{B} |u\pm (2\pi ia\, m-\pi i a)|^{B-A}.
\end{multline*}
The function $\Big|\tfrac{u}{u\pm (2\pi ia\, m-\pi i a)}\Big|\leq M_\pm$ is bounded on $Q_\pm$ (not true on the whole half-plane $\mathcal{H}_-$!) and the series $\sum_m  |u\pm (2\pi ia\, m-\pi i a)|^{B-A}$ converges since $B-A<-1$. This shows the estimate \eqref{stepest}.\\
\end{proof}
\begin{remark}
Note that formula \eqref{stepsumsf} for $F_\pm$ defines a
holomorphic function on the whole half plane $\mathcal{H}_-$. The
difference
\[
F_- - F_+ = \sum_{m\in\bbz} f(u+\pi i a + 2\pi i a\, m)
\]
defines a $2\pi i a$ periodic function on $\mathcal{H}_-$. However,
the estimate \eqref{stepest} does not extend to $\mathcal{H}_-$.
Passing to the variable $\tilde{h}=e^u/a$ the difference $(F_- -
F_+) (\tilde{h})$ defines a germ of a meromorphic function at the
orgin. This situation is in the spirit of the functional cochain \cite{ilyashenko}.
\end{remark}
\begin{corollary}
Using Lemmas \ref{lem:prf} and \ref{lem:sums} we can solve
explicitely the difference equation $\diff_a F=f$, where $f(\tfrac
1u,\tfrac{\log u}{u})$ is as in Lemma \ref{lem:prf}. Indeed, the
general solution consists of 3 terms: principal part, given by
$P\in\mathcal{P}(u)$, remainder given by series \eqref{stepsumsf}
and an arbitrary solution to the homogeneous equation $\diff_a F_H
\equiv 0$. The latter one is given by a series $\sum_l a_l e^{l
u/a}$.
\end{corollary}

%

%
%

In the next lemma we investigate the analytic properties of the
generalized compensator $\omega(h,\ep,\alpha)$ (see \eqref{omega})
for $\ep=0$. Recall that $\omega(h,\ep,\alpha=0)$ is the Roussarie
compensator. Below we study the case with $\ep=0$ and arbitrary
$\alpha$ in the logaritmic coordinate $u=\log h$. We denote
\begin{equation}
\label{wfn} w = - \frac1{\omega(e^u,0,\alpha)},
\end{equation}
so $w^{\alpha}e^{-1/w}=e^u$.
\begin{lemma}
For $\re u\ll 0$ we have
\label{lem:cmpan}
\begin{equation}
\label{anw} w=\tfrac 1u (a+ g_\alpha(\tfrac 1u, \tfrac{\log u}{u})),
\end{equation}
where $\bbc\ni a\neq 0$, $g_\alpha(\cdot,\cdot)$ is an analytic
function and $g_\alpha(0,0)=0$.
\end{lemma}

\begin{proof} Indeed, writing $w=-\frac {w_1} u$, we get
$$
z_1\alpha \log w_1-\alpha z_2+\frac 1{w_1}=1,\qquad z_1=\frac 1
u,\quad z_2=\frac{\log(-u)}{u}.
$$
The left-hand side of this equation is an analytic function
$F=F(w_1,z_1,z_2)$ in a neighborhood of $(1,0,0)$, and $F(1,0,0)=1$.
Since $\frac{\partial F}{\partial w_1}\mid_{(1,0,0)}=1$, by implicit
function theorem we get
$$
w=-\tfrac 1 u\left(1+g(\tfrac 1 u,\frac{\log(-u)}{u})\right).
$$
\end{proof}

\begin{proof}[Proof of Proposition \ref{pr:zersing}]
Note that the main difficulty in the proof is to control the form of the singularity of the function $I$ at $h=0$. Indeed, consider, as a toy example, the special case when $I$ is a meromorphic function of $h$. Then, the moderate growth bound \eqref{mdgrowth} restricts the order of pole at $h=0$ to $N$ and so the increment of argument satisfies \eqref{zerest}.
To prove a proposition in the general case it is enough to show that
the form of singularity which is allowed by the variation relation
\eqref{eq:thmvar} together with the moderate growth estimation
forces an explicit bound for the increment of argument in terms of
$N$ only. Due to this idea, it is enough to work pointwise with
respect to all parameters (i.e.\ $\ep,\alpha,\ldots$). The case
$\ep\neq 0$ was already investigated in \cite{BM}. The conclusion
was that the leading term of the integral $I(h)$ at $h=0$ is a
monomial $h^A \log^k h$, with positive, integer $k$. Thus, the same
estimate as in the meromorphic case holds.

First we give a proof in a special case $\alpha=0$ (compare
\eqref{eq:ea family}). It contains the essence of the general case
with much less technical details.

{\it The $\alpha=0$ case.} The function $w$ given by formula
\eqref{wfn} reads $w=-\tfrac{1}{\log h}$. We use the logarithmic
chart $u=\log h$. By Lemma \ref{lem:prf}, there exists a polynomial
$P\in\bbc[\log u,u,\tfrac 1u]$ (leading term) such that
\[
|F-\diff_{a_1,\ldots,a_k}P|\leq M |u|^{-A},
\]
for some $A>n$ and a positive constant $M$. Thus, the iterated
variation (difference) of $I-P$ is of sufficiently high order and a
solution $F_+$ defined in $Q_+$ is given by the iterated sum formula
\eqref{sumsf}. Moreover, it is of lower order then $P$.

Now, the iterated difference vanishes identically
\[
\diff_{a_1,\ldots,a_k} (I-P-F_+)\equiv 0.
\]
Thus, by Lemma 4.8 from \cite{BM}, the principal term of $I-P-F_+$
has the form $h^\alpha \log^m h$. Finally, the principal term of $I$
is either a monomial $h^\alpha \log^m h$, $\alpha\geq -N$,
$m\in\bbz$, $m\geq 0$ or $\log^lh\, \log^m(\log h)$, $m,l\in\bbz$,
$m\geq 0$. In both cases the upper bound \eqref{zerest} holds.
\medskip

{\it The general case ($\alpha\neq 0$)}. By Lemma \ref{lem:cmpan} we
know that the function $w$ has the following form
\[
w=\tfrac 1u (a+g(\tfrac 1u,\tfrac{\log u}{u})), \qquad a\neq 0,
\]
and $g$ is a holomorphic function, $g(0,0) = 0$. For arbitrary
meromorphic function $F(\cdot)$, the composition $F(w)$ has the
following expansion
\[
F(w) = \sum_{k\geq -k_0} (\tfrac 1u)^k q_k(\log u),
\]
where $q_k$ is a polynomial. Now we can repeat the argument used in the special case $\alpha = 0$. We take the principal part $P_F$ of $F(w)$ up to order $A>n$. It is a polynomial in $\log u$ and Laurent polynomial in $u$. We can solve the iterated difference equation explicitly, up to terms of higher order (Lemma \ref{lem:prf}). Then, by Lemma \ref{lem:sums}, a solution $F_+$ to the iterated difference equation for $(I-P_F)$ is given by the iterated sum formula \eqref{sumsf}. Finally, we obtain that the leading term of $I$ is a monomial $h^\alpha \log^m h$, $\alpha\geq -N$, $m\in\bbz$, $m\geq 0$ or $\log^l h\, \log^m(\log h)$, $m,l\in\bbz$, $m\geq 0$. In both cases the upper bound \eqref{zerest} holds.\\
\end{proof}
\begin{remark}
In the above proof one can replace the iterated sum solution $F_+$
by $F_-$, which is well defined over $Q_-$. The remaining part of
the proof works as well with $F_-$.
\end{remark}

\bibliographystyle{plain}

\end{document}